\documentclass[11pt, a4paper]{article}

\usepackage{amsmath}
\usepackage{amssymb}
\usepackage{amsthm}
\usepackage{appendix,xcolor}
\usepackage{verbatim}

\newtheorem{theorem}{Theorem}[section]
\newtheorem{corollary}{Corollary}[section]

\newtheorem{lemma}{Lemma}[section]
\newtheorem{remark}{Remark}[section]
\newtheorem{example}{Example}[section]
\newtheorem*{thmErdos}{Theorem A}
\newtheorem*{thmMarkus}{Theorem B}

\newcommand{\Ltwo}{\ensuremath{L^2(a, b)}}
\newcommand{\Clo}{\ensuremath{{\mathcal{H}_\Lambda}} }

\begin{document}
\title{Hereditary completeness of Exponential systems $\{e^{\lambda_n t}\}_{n=1}^{\infty}$ in their closed span in $L^2 (a, b)$
and Spectral Synthesis}
\author{Elias Zikkos\footnote{Department of Mathematics, Khalifa University, Abu Dhabi, United Arab Emirates,\\ 
email address:  elias.zikkos@ku.ac.ae and eliaszikkos@yahoo.com} and
Gajath Gunatillake\footnote{Department of Mathematics and Statistics, American University of Sharjah, Sharjah, United Arab Emirates,
\\email address: mgunatillake@aus.edu}
}

\maketitle

\begin{abstract}
Suppose that $\{\lambda_n\}_{n=1}^{\infty}$ is a sequence of distinct positive real numbers
satisfying the conditions inf$\{\lambda_{n+1}-\lambda_n \}>0,$ and
$\sum_{n=1}^{\infty}\lambda_n^{-1}<\infty.$
We prove that the exponential system $\{e^{\lambda_n t}\}_{n=1}^{\infty}$ is hereditarily complete in
the closure of the subspace spanned by $\{e^{\lambda_n t}\}_{n=1}^{\infty}$ in the space $L^2 (a,b)$.
We also give an example of a class of compact non-normal operators defined on this closure which admit spectral synthesis.
\end{abstract}



\section{Introduction}
\setcounter{equation}{0}

Let $(a,b)$ be a bounded interval on the real line and let $L^2 (a,b)$
be the space of square integrable complex valued functions on $(a,b)$,
equipped with the norm
\[
||f||_{L^2 (a,b)}:=\left(\int_{a}^{b}|f(x)|^2\, dx\right)^{\frac{1}{2}},
\]
and endowed with the inner product
\[
\langle f,g \rangle:=\int_{a}^{b} f(x)\overline{g(x)}\, dx.
\]
Let $\Lambda$ denote the sequence of distinct positive real numbers $\{\lambda_n\}_{n=1}^{\infty}.$ Assume that
\begin{equation}\label{LKcondition}
\sum_{n=1}^{\infty}\frac{1}{\lambda_n}<\infty,
\qquad\text{and}\qquad  \lambda_{n+1}-\lambda_n\ge c,
\end{equation}
for some positive constant $c$ and all $n.$
Consider the associated exponential system $\{e^{\lambda_n t}\}_{n\in\mathbb{N}}$ of $\Lambda$.
We will denote this system by $E_{\Lambda}$.
Let $\Clo (a,b)$ denote the closed span of $E_{\Lambda}$ in $\Ltwo.$

Since the series $\sum \lambda_n^{-1}$ converges, it follows from the M\"{u}ntz-Sz\'{a}sz theorem that
$E_{\Lambda}$ is $not\,\, complete$ in $\Ltwo$, that is, $\Clo (a,b)$ is a  proper subspace of $\Ltwo$
(see \cite[Theorem 6.1]{LK}, \cite[Theorem 4.2.6]{BE}, and \cite{GurariyLusky}).
Moreover, $E_{\Lambda}$ is a $minimal$ system, that is, each function $e^{\lambda_n t}$
does not belong to the closed span of the remaining vectors of $E_{\Lambda}$ in $\Ltwo$
(see \cite[relation (1.9)]{LK}).

It is known that a family $\{f_n\}$ of functions in a separable Hilbert space $\cal{H}$ is minimal, if and only if
it has a $biorthogonal$ sequence in $\cal{H}$ (see \cite[Lemma 3.3.1]{Christensen}): in other words,
there exists a sequence $\{g_n\}_{n\in \mathbb{N}}\subset\cal{H}$ so that
\[
\langle g_n, f_m\rangle =\begin{cases} 1, & m=n, \\  0, & m\not=n.\end{cases}
\]
We note that a system in $\cal H$ which is both complete and minimal is called $exact$. An exact system has a unique biorthogonal family
in $\cal H$, however this family is not necessarily complete in $\cal H$ (see \cite{YoungP}).

Clearly now the system $E_{\Lambda}$ has a unique biorthogonal family
$\{g_n\}_{n\in \mathbb{N}}$ in the space $\Clo (a,b)$, assuming relation $(\ref{LKcondition})$.
We will show that this family is complete in $\Clo (a,b)$, thus
the system $E_{\Lambda}$ is a $Markushevich\,\, basis$ in $\Clo (a,b)$. In fact, we will prove that
$E_{\Lambda}$ is a $strong\,\, Markushevich\,\, basis$ in $\Clo (a,b)$ (Theorem $\ref{hereditary}$).
This means that
for any disjoint union of two sets $N_1$ and $N_2$, such that
$\mathbb{N}=N_1\cup N_2$, the closed span of the mixed system
\[
\{e^{\lambda_n t}:\,\, n\in N_1\}\cup\{g_{n}:\,\, n\in N_2\}
\]
in $\Ltwo$ is equal to $\Clo (a,b)$. We point out that the phrase ``strong Markushevich basis''
is used interchangeably with the phrase $hereditarily\,\, complete\,\, system$.

\begin{theorem}\label{hereditary}
Let the sequence $\Lambda=\{\lambda_n\}_{n=1}^{\infty}$ satisfy the condition $(\ref{LKcondition})$.
Then the exponential system $E_{\Lambda}$ is $\bf{hereditarily\,\, complete}$ in the space $\Clo (a,b)$.
\end{theorem}

Our work has been inspired by \cite{Baranov2013} where it was proved that a complete and minimal exponential system
$\{e^{i\lambda_n t}:\,\, n\in \mathbb{N}\}$ in $L^2(-a,a)$ is hereditarily complete up to one-dimensional effect
(\cite[Theorem 1.1]{Baranov2013}). That is, if $\{g_{n}:\,\, n\in \mathbb{N}\}$ is the unique biorthogonal family
to the system $\{e^{i\lambda_n t}:\,\, n\in \mathbb{N}\}$ in $L^2(-a,a)$, then
for any disjoint union of two sets $N_1$ and $N_2$, such that
$\mathbb{N}=N_1\cup N_2$, the orthogonal complement in $L^2(a, b)$ to the closed span of the mixed system
\[
\{e^{i\lambda_n t}:\,\, n\in N_1\}\cup\{g_{n}:\,\, n\in N_2\},
\]
is at most one dimensional.
The latter phenomenon is possible since
there exists an exact exponential system $\{e^{i\lambda_n t}\}$ in $L^2 (-\pi,\pi)$ which is not hereditarily complete
(\cite[Theorem 1.3]{Baranov2013}).

\begin{remark}
We point out that our exponential system $E_{\Lambda}$ is not complete in $L^2(a,b)$ but it is hereditarily complete in $\Clo (a,b)$.
\end{remark}

The notion of hereditary completeness is closely related to the $Spectral\,\, Synthesis$ problem for linear operators
\cite{Markus1970,Baranov2013,Baranov2015,Baranov2022}.
We say that a bounded linear operator $T$ in a separable Hilbert space $H$ having a set of eigenvectors which is complete in $H$,
admits $Spectral\,\, Synthesis$ if for any invariant subspace $A$ of $T$, the set of eigenvectors of $T$ contained in $A$
is complete in $A$.
It follows by the results obtained by J. Wermer in \cite{Wermer1952} that an operator which is $\bf both$ compact and normal
admits spectral synthesis. However, if one of the two conditions does not hold, then the conclusion might not be true.
In particular, Wermer gave an example of a bounded normal operator which does not admit spectral synthesis \cite[Theorem 2]{Wermer1952}, while
A. S. Markus, amongst others, proved that there exist compact operators which do not admit spectral synthesis
(see \cite[Theorem 4.2]{Markus1970}).

Markus \cite[Theorem 4.1]{Markus1970} also showed that a compact operator with a trivial kernel and non-zero simple eigenvalues with corresponding eigenvectors $\{f_n\}$, admits spectral synthesis if and only if the family $\{f_n\}$ is hereditarily complete in $\cal H$.
Motivated by this result, and since our system $E_{\Lambda}$ is hereditarily complete in $\Clo (a,b)$,
we present in Section 4 a class of compact but not normal operators on $\Clo (a,b)$ which admit spectral synthesis.

\section{Properties of the space $\Clo (a,b)$}
\setcounter{equation}{0}

Subject to the condition $(\ref{LKcondition})$,
Luxemburg and Korevaar proved in \cite[Theorem 8.2]{LK} that a function
$f$ belongs to the space $\Clo (a,b)$ if and only if
$f(x)=g(x)$ almost everywhere on $(a,b)$, where $g$ is an analytic function in the half plane $\Re z <b$,
admitting the Dirichlet series representation
\[
g(z)=\sum_{n=1}^{\infty} a_n e^{\lambda_n z}\qquad a_n\in\mathbb{C},\quad \forall \,\,z\in \Re z<b,
\]
with the series converging uniformly on compact subsets of the half plane $\Re z <b$.

Let us rewrite the above result using another notation.
Consider the space of square integrable functions on $(a,b)$
that admit a Dirichlet series representation
\[
f(z)=\sum_{n=1}^{\infty} c_{n} e^{\lambda_n z}
\]
in the half-plane $\Re z<b$, converging uniformly on compacta.
Denote this space by $L^2 (a, b, E_{\Lambda})$, that is
\[
L^2 (a, b, E_{\Lambda}):=\left\{f\in L^2 (a, b):\,\, f(z)=\sum_{n=1}^{\infty}
c_{n} e^{\lambda_n z}\quad \forall\,\, z\in\Re z<b\right\}.
\]

The Luxemburg-Korevaar result is now restated as follows.
\begin{thmErdos}
Let $\{\lambda_n\}_{n=1}^{\infty}$ be a sequence of numbers satisfying the condition $(\ref{LKcondition})$.
Then
\[
\Clo (a,b)=L^2 (a, b, E_{\Lambda}).
\]
\end{thmErdos}

The above result depends heavily on deriving the distance lower bound $(\ref{distanceresult})$:
let $e_{n}(x):=e^{\lambda_n x}$ and let $E_{\Lambda_{n}}:=E_{\Lambda}\setminus \{e_{n}\}$,
that is $E_{\Lambda_{n}}$ is the system $E_{\Lambda}$ excluding the function $e^{\lambda_n x}$.
Denote by $D_{n}$ the $\bf Distance$ between $e_{n}$
and the closed span of $E_{\Lambda_{n}}$ in $L^2(a,b)$,
\[
D_{n}:=\inf_{g\in \overline{\text{span}} (E_{\Lambda_{n}})} ||e_{n}-g||_{L^2 (a,b)}.
\]
It was shown (\cite[relation (1.9)]{LK}) that for every $\epsilon>0$, there is a positive constant $m_{\epsilon}$
which does not depend on $n\in\mathbb{N}$, but only on $\Lambda$ and the segment $(a,b)$, so that
\begin{equation}\label{distanceresult}
D_n\ge m_{\epsilon}\cdot e^{(b-\epsilon)\lambda_n}.
\end{equation}

\section{Hereditary Completeness: Proof of Theorem $\ref{hereditary}$}
\setcounter{equation}{0}

In this section we prove our hereditary completeness result by first obtaining the Fourier type series representation
$(\ref{representationf})$ for functions in $\Clo (a,b)$.

\begin{lemma}\label{hereditarylemma}
Let the sequence $\Lambda=\{\lambda_n\}_{n=1}^{\infty}$ satisfy the condition $(\ref{LKcondition})$ and
consider the space $\Clo (a,b)$. Then the following hold.

(I) There exists a family of functions
\[
r_{\Lambda}=\{r_{n}:\,\, n\in\mathbb{N}\}\subset \Clo (a,b)
\]
so that it is the unique biorthogonal sequence to $E_{\Lambda}$ in $\Clo (a,b)$.
For every $\epsilon>0$ there is a constant $m_{\epsilon}>0$, independent of $n\in\mathbb{N}$, so that
\begin{equation}\label{rnkbound}
||r_{n}||_{L^2 (a,b)} \le  m_{\epsilon}e^{(-b+\epsilon)\lambda_n},\qquad \forall\,\, n\in\mathbb{N}.
\end{equation}

$(II)$ For each $f$ in $\Clo (a,b)$
there exists a function $g$ analytic in the half-plane $\Re z<b$,
so that
\[
f(x)=g(x)\qquad \text{almost everywhere on}\quad (a,b),
\]
with $g$ admitting the Fourier-type Dirichlet series representation
\begin{equation}\label{representationf}
g(z)=\sum_{n=1}^{\infty}\langle f, r_{n} \rangle \cdot e^{\lambda_n z},
\end{equation}
converging uniformly on compact subsets of the half-plane $\Re z<b$.
\end{lemma}

\begin{proof}

$\quad$

$(I)$ Recall the distance result $(\ref{distanceresult})$ and the system  $E_{\Lambda_{n}}$.
Now, since $L^2(a,b)$ is a Hilbert space, it then follows that there exists a unique element
in $\overline{\text{span}}(E_{\Lambda_{n}})$ in $L^2(a,b)$,that we denote by $\Phi_{n}$,
so that
\[
||e_{n}-\Phi_{n}||_{L^2 (a,b)}=
\inf_{g\in \overline{\text{span}}(E_{\Lambda_{n}})}||e_{n}-g||_{L^2 (a,b)}=D_{n}.
\]
The function $e_{n}-\Phi_{n}$ is orthogonal to all the elements of the closed span of $E_{\Lambda_{n}}$
in $L^2 (a,b)$, hence to $\Phi_{n}$ itself.
Therefore
\[
\langle e_{n}-\Phi_{n}, e_{n}-\Phi_{n} \rangle=\langle e_{n}-\Phi_{n}, e_{n}\rangle.
\]
Hence
\[
(D_{n})^2=\langle e_{n}-\Phi_{n}, e_{n}\rangle.
\]
Next, we define
\[
r_{n}(x):=\frac{e_{n}(x)-\Phi_{n}(x)}{(D_{n}) ^2},\qquad\text{thus}\qquad ||r_{n}||_{L^2 (a,b)}=\frac{1}{D_{n}}.
\]
It then follows that $\langle r_{n}, e_{n}\rangle=1$ and $r_{n}$ is orthogonal to all the
elements of the system $E_{\Lambda_{n}}$.
Thus $\{r_{n}:\,\, n\in\mathbb{N}\}$ is biorthogonal  to the system $E_{\Lambda}$.
Since $\Phi_{n}\in\overline{\text{span}}(E_{\Lambda_{n}})$ in $L^2 (a,b)$ it follows that
$r_{n}\in\overline{\text{span}}(E_{\Lambda})$ in $L^2 (a,b)$.
Moreover, since $||r_{n}||_{L^2 (a,b)}=1/D_{n}$, then from $(\ref{distanceresult})$ we obtain $(\ref{rnkbound})$.\\

\smallskip

$(II)$ From Theorem $A$ we know that if $f$ belongs to $\Clo (a,b)$ then $f$
is extended analytically to the half-plane $\Re z<b$, admitting the series representation
$f(z)=\sum_{m=1}^{\infty} c_{m} e^{\lambda_m z}$ for some coefficients $c_m\in\mathbb{C}$, converging uniformly on compacta.
We will show below that
\[
c_{n}=\langle f, r_{n} \rangle\qquad \text{for\,\, all}\quad n\in\mathbb{N}.
\]

Fix some $n\in\mathbb{N}$. Then
\begin{eqnarray}
\langle f, r_{n} \rangle & = &
\int_{a}^{b}\overline{r_{n}(x)} \cdot \left(\sum_{m=n+1}^{\infty} c_{m} e^{\lambda_m x}\right)\, dx +
\int_{a}^{b}\overline{r_{n}(x)} \cdot \left(\sum_{m=1}^{n} c_{m} e^{\lambda_m x}\right) \, dx
\nonumber\\
& = &
\int_{a}^{b} \overline{r_{n}(x)}\cdot\left(\sum_{m=n+1}^{\infty} c_{m} e^{\lambda_m x}\right)\, dx + c_{n}\label{cnk}
\end{eqnarray}
with the last step valid due to the biorthogonality.

We claim that
\begin{equation}\label{zerozero1}
\int_{a}^{b} \overline{r_{n}(x)}\cdot \left(\sum_{m=n+1}^{\infty} c_{m} e^{\lambda_m x}\right)\, dx=0.
\end{equation}

Indeed, since $f\in L^2(a,b)$, then the Dirichlet series
\[
Q_n(z):=f(z)-\sum_{m=1}^{n} c_{m} e^{\lambda_m z}=\sum_{m=n+1}^{\infty} c_{m} e^{\lambda_m z}
\]
also belongs to $L^2(a,b)$. It then follows from Theorem $A$ that $Q_n$ belongs
to the closed span of the exponential system
\[
E_{\Lambda (n+1)}:=E_{\Lambda}\setminus\{e^{\lambda_m x}:\, m=1,2,\dots,n\}=\{e^{\lambda_m x}:\, m\ge n+1\}
\]
in $L^2(a,b)$. Thus, for every $\epsilon>0$, there is a function $f_{\epsilon}$ in the span of $E_{\Lambda (n+1)}$
so that $||Q_n-f_{\epsilon}||_{L^2(a,b)}<\epsilon$. From biorthogonality we have
\[
\int_{a}^{b} \overline{r_{n}(x)}\cdot f_{\epsilon}(x)\, dx=0.
\]
Combining this with the Cauchy-Schwarz inequality we get
\begin{eqnarray*}
\left|\int_{a}^{b} \overline{r_{n}(x)}\cdot Q_n(x)\, dx\right| & = &
\left|\int_{a}^{b} \overline{r_{n}(x)}\cdot (Q_n (x)-f_{\epsilon}(x))\, dx\right|\\
& \le & \epsilon\cdot ||r_{n}||_{L^2(a,b)}.
\end{eqnarray*}
The arbitrary choice of $\epsilon$ implies that $(\ref{zerozero1})$ is true.

Finally, replacing $(\ref{zerozero1})$ in $(\ref{cnk})$
shows that $\langle f, r_{n} \rangle=c_n$, that is  $(\ref{representationf})$ holds.

\end{proof}

\subsection{Proof of Theorem $\ref{hereditary}$}

Let us write the set $\mathbb{N}$ as an arbitrary disjoint union of two sets $N_1$ and $N_2$.
In order to obtain the hereditary completeness of $E_{\Lambda}$ in $\Clo (a,b)$,
we must show that the closed span of the mixed system
\[
E_{1,2}:=\{e_{n}:\,\, n\in N_1\}\cup\{r_{n}:\,\, n\in N_2\},
\]
in $L^2 (a, b)$ is equal to $\Clo (a,b)$.

Denote by $W_{\Lambda_{1,2}}$ the closed span of $E_{1,2}$ in $L^2 (a, b)$.
Obviously $W_{\Lambda_{1,2}}$ is a subspace of $\Clo (a,b)$.
Let $W^{\perp}_{\Lambda_{1,2}}$ be the orthogonal complement of $W_{\Lambda_{1,2}}$ in $\Clo (a,b)$, that is
\[
W^{\perp}_{\Lambda_{1,2}}=\{f\in \Clo (a,b):\,\, \langle f, g \rangle =0\quad \text{for\,\, all}\,\, g\in W_{\Lambda_{1,2}}\}.
\]
Now, if $f\in W^{\perp}_{\Lambda_{1,2}}\subset \Clo (a,b)$ then by $(\ref{representationf})$ we have
\[
f(t)=\sum_{n=1}^{\infty} \langle f, r_{n} \rangle \cdot e^{\lambda_n t},
\quad \text{almost everywhere on}\,\, (a,b)
\]
and of course $\langle f, r_{n} \rangle =0$ for all $n\in N_2$.
Thus,
\[
f(t)=\sum_{n\in N_1}\langle f, r_{n} \rangle \cdot e^{\lambda_n t},
\quad \text{almost everywhere on}\,\, (a,b).
\]

Since the above Dirichlet series $f$ is analytic in the half-plane $\Re z<b$
and belongs to $L^2 (a, b)$, it follows from Theorem $A$ that $f$ belongs to the closed span of the exponential system
\[
E_{\Lambda, N_1}:=\{e_{n}:\,\, n\in N_1\}
\]
in $L^2 (a, b)$. Hence, for every $\epsilon>0$ there is a function $g_{\epsilon}$ in $\text{span}(E_{\Lambda, N_1})$ so that
$||f-g_{\epsilon}||_{L^2(a,b)}<\epsilon$. Write
\[
\langle f, f \rangle = \langle f, f-g_{\epsilon} \rangle + \langle f, g_{\epsilon} \rangle.
\]
Since $f\in W^{\perp}_{\Lambda_{1,2}}$ then $\langle f, e_{n} \rangle =0$ for all $n\in N_1$, thus
$\langle f, g_{\epsilon} \rangle=0$. Therefore,
\[
||f||^2_{L^2(\gamma,\beta)}=\langle f, f \rangle = \langle f, f-g_{\epsilon}\rangle \le ||f||_{L^2(a,b)}\cdot ||f-g_{\epsilon}||_{L^2(a,b)}
\le ||f||_{L^2(a,b)}\cdot \epsilon.
\]
Hence
\[
||f||_{L^2(a,b)}\le \epsilon.
\]
Clearly this means that $f(x)=0$ almost everywhere on $(a, b)$, hence $W^{\perp}_{\Lambda_{1,2}}=\{\mathbf{0}\}$.
Thus $W_{\Lambda_{1,2}}=\Clo (a,b)$, meaning that the exponential system $E_{\Lambda}$
is hereditarily complete in the space $\Clo (a,b)$.
The proof of Theorem $\ref{hereditary}$ is now complete.

\section{Compact operators on $\Clo (a,b)$ admitting Spectral Synthesis}
\setcounter{equation}{0}

Our goal in this section is to present a class of compact but not normal operators on $\Clo (a,b)$ that admit Spectral Synthesis.
The operators constructed will have the properties presented below, a result obtained by Markus.

\begin{thmMarkus}\label{Compact}\cite[Theorem 4.1]{Markus1970}

Let $T:\cal{H}$$\to\cal{H}$ be a compact operator such that

(i) its kernel is trivial and

(ii) its non-zero eigenvalues are $\bf simple$.

Let $\{f_n\}_{n\in\mathbb{N}}$ be the corresponding sequence of eigenvectors.
Then $T$ admits $\bf{Spectral\,\, Synthesis}$ if and only if $\{f_n\}_{n\in\mathbb{N}}$ is $\bf Hereditarily\,\, complete$ in $\cal{H}$.
\end{thmMarkus}

\subsection{A class of operators on $\Clo (a,b)$}

Let $f$ be a function in the space $\Clo (a,b)$. Then from Lemma $\ref{hereditarylemma}$ we have
\[
f(z)=\sum_{n=1}^{\infty} \langle f , r_n\rangle\cdot  e_n(z),
\]
with the series converging uniformly on compact subsets of the half-plane $\Re z<b$.
From relation $(\ref{rnkbound})$, for every $\epsilon>0$ there exists some $m_{\epsilon}>0$, independent of $n\in\mathbb{N}$, so that
\[
|\langle f , r_n\rangle|\le ||f||_{L^2 (a,b)}\cdot m_{\epsilon}e^{(-b+\epsilon)\lambda_n}.
\]

Let $\delta>0$ and choose a sequence $\{u_n\}_{n=1}^{\infty}$ of distinct non-zero complex numbers such that
\begin{equation}\label{un}
|u_n|\le  e^{-\delta\lambda_n}.
\end{equation}
Clearly $u_n\to 0$ as $n\to\infty$.
\begin{example}\label{example}
For $\delta>0$, we may take $u_n= e^{-\delta\lambda_n}$.
\end{example}

It follows from above that for every $0<\epsilon<\delta$,
there exists some $m_{\epsilon}>0$, independent of $n\in\mathbb{N}$ and $f\in \Clo (a,b)$, so that

\begin{equation}\label{upperbound}
|\langle f , r_n\rangle\cdot u_n|\le ||f||_{L^2 (a,b)}\cdot m_{\epsilon}e^{(-b-\delta+\epsilon)\lambda_n}.
\end{equation}

Now, define
\begin{equation}\label{operator}
T(f(z)):=\sum_{n=1}^{\infty} \langle f , r_n\rangle\cdot u_n\cdot  e^{\lambda_n z}.
\end{equation}

One deduces from $(\ref{upperbound})$ that $T(f(z))$
is a function analytic in the half-plane $\Re z<b+\delta$, converging uniformly on compact subsets of this half-plane, thus on
the interval $[a,b]$ as well. The uniform convergence on $[a,b]$ implies that

$(i)$ $T(f(z))$ belongs to the space $\Clo (a,b)$.

$(ii)$ The series $T(f)$ converges in the $L^2(a,b)$ norm.

\begin{remark}
It also follows from $(\ref{upperbound})-(\ref{operator})$ that there exists some $N>0$ so that
\[
||T(f)||_{L^2(a,b)}\le N||f||_{L^2(a,b)}\quad \text{for\,\, all}\quad f\in \Clo (a,b).
\]
Therefore, $T:\Clo (a,b)\to \Clo (a,b)$ defines a Bounded Linear Operator.
We denote by $T^*$  its Adjoint operator.
\end{remark}

We will show that
\begin{theorem}\label{spectralsynthesis}
The operator $T:\Clo (a,b)\to \Clo (a,b)$ defined in $(\ref{operator})$ is compact, not normal, and admits spectral synthesis.
\end{theorem}

\begin{proof}
Our result follows from Lemmas $\ref{compactoperator}-\ref{notnormal}$, Theorem $B$, and the fact that the system $E_{\Lambda}$ is hereditarily complete in $\Clo (a,b)$.
\end{proof}

We remark that if we choose $u_n= e^{-\delta\lambda_n}$ as in Example $\ref{example}$, then $T(f(z))$ is equal to the translation $f(z-\delta)$,
which is well defined on $\Clo (a,b)$ since any function in this space extends analytically to the half-plane $\Re z<b$. Thus we have

\begin{corollary}
For any $\delta>0$, the shift operator $T_{\delta}(f):=f(z-\delta)$  defined on $\Clo (a,b)$ 
is compact, not normal, and admits spectral synthesis.
\end{corollary}

\subsection{Auxiliary results: Lemmas $\ref{compactoperator}$, $\ref{Eigenvalues}$, $\ref{notnormal}$, and $\ref{TheAdjoint}$}

In the following four lemmas, we prove that the operator $T$ is compact, we obtain various properties of $T$
as those stated in Theorem $B$, and we show that $T$ is not normal.
Finally, a Fourier type series representation is derived for the adjoint $T^*$.  \\

\smallskip

If $L$ is a bounded operator on $\Clo (a,b)$, we use $\| L \|$ to
denote the operator norm of $L$ which is the supremum of the set
$ \{ \| L(f)\|_{L^2(a,b)}: f\in \Clo (a,b), \| f\|_{L^2(a,b)}=1 \}$.

\begin{lemma}\label{compactoperator}
The operator $T$ is compact.
\end{lemma}
\begin{proof}

Let $e_n(x)=e^{\lambda_n x}.$
Define  $T_m$ on $\Clo (a,b)$ by
\[
T_m(g)(x)=\sum_{n=1}^m \langle g, r_n \rangle u_n e_n(x).
\]
Let $f$ be a unit vector in $\Clo (a,b)$. It is easy to see that
\[
\|(T-T_m)(f)\|_{L^2(a,b)}\leq \sum_{n=m+1}^\infty \| \langle f, r_n \rangle u_n e_n  \|_{L^2(a,b)}.
\]

We have that $\| \langle f, r_n \rangle u_n e_n  \|_{L^2(a,b)}\leq m_\epsilon e^{(-b-\delta+\epsilon)\lambda_n} \| e_n \|_{L^2(a,b)}.$
It is easy to see that $\| e_n \|_{L^2(a,b)}<e^{\lambda_nb}.$
Thus,  $\| \langle f, r_n \rangle u_n e_n  \|_{L^2(a,b)}\leq m_\epsilon e^{(-\delta+\epsilon)\lambda_n}.$ Therefore,
\[
\|(T-T_m)(f)\|_{L^2(a,b)}\leq m_\epsilon  \sum_{n=m+1}^\infty e^{(-\delta+\epsilon)\lambda_n}.
\]
Since $\delta > \epsilon,$ the series above converges. Hence, $\|T-T_m\|$ tends to zero as $m$ tends to infinity.
Thus, the finite rank operators $\{ T_m\}$ converge to $T$ in the uniform operator topology. Therefore,
$T$ is compact.
\end{proof}

\begin{lemma}\label{Eigenvalues}
The following are true about the operator $T$ and its adjoint.
\begin{enumerate}
\item $\{ u_k\}_{k=1}^{\infty}$ are eigenvalues  of $T$ and $\{ e_k\}_{k=1}^{\infty}$ are the corresponding eigenvectors.
  \item $\{\overline{u_k}\}$ are eigenvalues  of $T^*$ and $\{ r_k\}_{k=1}^{\infty}$ are the corresponding eigenvectors.
  \item The kernel of $T$ is trivial.
  \item The spectrum of $T$ is $\displaystyle{\{ 0\} \cup \{ u_k\}_{k=1}^{\infty}}$.
  \item Each eigenvalue of $T$ is simple.

\end{enumerate}
\end{lemma}
\begin{proof}

$\quad$

1. Due to the biorthogonal relation between $E_{\Lambda}$ and $r_{\Lambda}$, we see in $(\ref{operator})$ that
\[
T(e_k)=u_k e_k.
\]

2. For fixed $k\in\mathbb{N}$ and any $n\in\mathbb{N}$ we have
\begin{eqnarray}
\langle T^*r_k-\overline{u_k} r_k, e_n\rangle & = & -\overline{u_k}\langle r_k, e_n\rangle +\langle r_k, Te_n\rangle\nonumber\\
& = & -\overline{u_k}\langle r_k, e_n\rangle + \langle r_k, u_n e_n\rangle\nonumber\\
& = & -\overline{u_k}\langle r_k, e_n\rangle + \overline{u_n}\langle r_k, e_n\rangle. \label{zero}
\end{eqnarray}
If $n=k$ then $(\ref{zero})$ equals zero, and the same holds for all $n\not= k$
due to the biorthogonal relation between $r_\Lambda$ and $E_\Lambda$ .
Thus $\langle T^*r_k-\overline{u_k} r_k, e_n\rangle=0$ for all $n\in\mathbb{N}$, hence
\[
T^*r_k-\overline{u_k}\cdot r_k=\mathbf{0}.
\]

3. We now show that the kernel of $T$ is the zero set.
For any $n\in\mathbb{N}$ we have
\[
\langle Tf, r_n\rangle =\langle f, T^*r_n\rangle = \langle f, \overline{u_n} r_n\rangle.
\]
If $Tf=0$ then one has $\langle f, r_n\rangle =0$ for all $n\in\mathbb{N}$. Since the family $\{r_n\}_{n=1}^{\infty}$
is complete in the space $\Clo (a,b)$, then $f=\mathbf{0}$. Hence the kernel of $T$ is the zero function.\\

\smallskip

4. Next we prove that $\{u_k\}$ are the only non-zero eigenvalues.
Suppose that
$Tf=\lambda f$ for some $\lambda\notin\{u_n\}_{n=1}^{\infty}$ and $f\not=\mathbf{0}$. Then,
\begin{align}
\lambda\langle f, r_n\rangle& =  \langle Tf, r_n\rangle, \notag \\
& =\langle f, T^*r_n\rangle \notag, \\
&= \langle f, \overline{u_n} r_n\rangle, \notag \\
&=u_n\langle f, r_n\rangle.\label{Only_eigenvalues}
\end{align}
Thus,
\[
(\lambda-u_n)\cdot \langle f, r_n\rangle=0
\]
for all $n.$ But $\lambda\notin\{u_n\}_{n=1}^{\infty}$ therefore $\langle f, r_n\rangle=0$ for all $n\in\mathbb{N}$.
Since the family $\{r_n\}_{n=1}^{\infty}$ is complete in the space $\Clo (a,b)$, then $f=\mathbf{0}$.
This contradiction shows that $\{ u_k\}$ are the only non-zero eigenvalues.
Since $T$ is compact it follows that the spectrum of $T$ is
\[
\displaystyle{\{ 0\} \cup \{ u_k\}_{k=1}^{\infty}}.
\]

5. Next we prove that each eigenvalue of $T$ is simple.
Suppose that $Tf=u_k f$ for some $u_k\in\{u_n\}_{n=1}^{\infty}$ and some $f\in \Clo (a,b)$.
Then, using the same computation in (\ref{Only_eigenvalues}) above, we get

\[
(u_k-u_n)\langle f , r_n\rangle = 0 \qquad 
\]
for all $n.$ But if $n\not= k,$ then $u_n\not= u_k$, thus, $\langle f , r_n\rangle = 0$ for all $n\not= k$. Since
\[
f(t)=\sum_{n=1}^{\infty} \langle f , r_n\rangle\cdot  e_n(t),
\]
then we get $f(t)=\langle f , r_k\rangle e_k$. But this means that $u_k$ is simple.

\end{proof}

\begin{lemma}\label{notnormal}
The operator $T$ is not normal.
\end{lemma}
\begin{proof}
Suppose that $T$ is a normal operator, thus
\[
TT^*(e_k)=T^*T(e_k)\qquad \text{for\,\,all}\quad k\in\mathbb{N}.
\]
But any two eigenvectors that correspond to different eigenvalues of a normal operator are orthogonal.
Clearly this is not the case here for the set of eigenvectors $e_n=e^{\lambda_n t}$ of $T$.
\end{proof}

And finally, biorthogonality of $\{ e_n\}$ and $\{r_n \}$ allows us to compute $T^*$ explicitly.

\begin{lemma}\label{TheAdjoint}
The adjoint operator $T^*$ of $T$ admits the representation
\begin{equation}\label{adjoint}
T^*f(z)=\sum_{n=1}^{\infty} \langle f , e_n\rangle\cdot \overline{u_n}\cdot  r_n(z),
\end{equation}
with the series converging uniformly on compact susbets of the half-plane $\Re z<b$. The series also converges in the
$L^2 (a,b)$ norm.
\end{lemma}

\begin{proof}
Given $f\in \Clo (a,b)$,

(I) We first prove that the series
\begin{equation}\label{seriesgf}
g_f(z):=\sum_{n=1}^{\infty} \langle f , e_n\rangle\cdot \overline{u_n}\cdot  r_n(z)
\end{equation}
converges uniformly on compact susbets of the half-plane $\Re z<b$.

$(II)$ We then show that $g_f$ belongs to the space $\Clo (a,b)$ and  $g_f$ converges in the
$L^2 (a,b)$ norm.

$(III)$ And finally we prove that
\[
\langle Th, f\rangle = \langle h, g_f\rangle\qquad \text{for\,\, all}\quad h\in \Clo (a,b).
\]

\smallskip

$(I)$ Consider first an arbitrary compact subset $K$ of the half plane $\Re z<b$,
such that $\Re z\le A$ for all $z\in K$ for some $A<b$.  From the Fourier-type Dirichlet series representation
$(\ref{representationf})$, we have
\[
r_n(z)=\sum_{m=1}^{\infty} \langle r_n, r_m \rangle e^{\lambda_m z}
\]
with the series converging uniformly on $K$. Also, from $(\ref{rnkbound})$,
for every $\epsilon>0$ there exists a positive constant $M_{\epsilon}$ so that
\[
|\langle r_n, r_m \rangle|\le M_{\epsilon}\cdot e^{(-b+\epsilon)\lambda_n}\cdot e^{(-b+\epsilon)\lambda_m}.
\]
Take $0<\epsilon<b-A$ thus $-b+\epsilon+A<0$. Then
\[
|r_n(z)|\le M_{\epsilon}\cdot e^{(-b+\epsilon)\lambda_n}\cdot \sum_{m=1}^{\infty}e^{(-b+\epsilon)\lambda_m}\cdot e^{A\lambda_m}
\quad\text{for\,\, all}\quad z\in K.
\]
Since $-b+\epsilon+A<0$, this means that there is some positive constant $P_K$ which depends on the set $K$,
and not on $n\in\mathbb{N}$, so that
\[
|r_n(z)|\le P_K\cdot M_{\epsilon}\cdot e^{(-b+\epsilon)\lambda_n}\quad \text{for\,\, all}\quad z\in K.
\]
Now, from $(\ref{un})$, the Cauchy-Schwarz inequality and since $e_n(t)=e^{\lambda_n t}$, one gets
\[
|\langle f, e_n \rangle \cdot \overline{u_n}|\le ||f||_{L^2 (a,b)}\cdot e^{\lambda_n b}\cdot e^{-\delta\lambda_n}.
\]
Combining the above bounds shows that
\[
|\langle f, e_n \rangle \cdot \overline{u_n}\cdot r_n(z)|\le P_K\cdot M_{\epsilon}\cdot ||f||_{L^2 (a,b)}\cdot e^{(-\delta+\epsilon)\lambda_n}
\quad \text{for\,\, all}\quad z\in K.
\]
Since $\epsilon<\delta$, the above shows that the series $g_f$ $(\ref{seriesgf})$ converges uniformly on $K$. \\

\smallskip

$(II)$ Next, for every $n\in\mathbb{N}$ let
\[
V_n(z):= \langle f, e_n \rangle \cdot \overline{u_n}\cdot r_{n}(z),
\]
thus
\[
g_f(z)=\sum_{n=1}^{\infty} V_n(z).
\]
It then follows from relations $(\ref{rnkbound})$ and $(\ref{un})$, that for every $\epsilon\in (0, \delta)$,
there is a positive constant $M_{\epsilon}$, so that
\begin{eqnarray*}
||V_n(t)||_{L^2 (a,b)} & = &
|\langle f, e_n \rangle| \cdot |u_n|\cdot ||r_{n}||_{L^2 (a,b)}\\
& \le & M_{\epsilon}\cdot ||f||_{L^2 (a,b)}\cdot e^{\lambda_n b}\cdot e^{-\delta\lambda_n}\cdot e^{(-b+\epsilon)\lambda_n}.
\end{eqnarray*}
Thus
\[
||V_n(t)||_{L^2 (a,b)}\le  M_{\epsilon}\cdot ||f||_{L^2 (a,b)}\cdot e^{(-\delta+\epsilon)\lambda_n}.
\]
Since $\epsilon<\delta$ then
\begin{equation}\label{Vnt}
\sum_{n=1}^{\infty}||V_n(t)||_{L^2 (a,b)}<\infty.
\end{equation}

Consider now the function
\begin{equation}\label{H}
H(t):=\sum_{n=1}^{\infty}|V_n(t)|=\sum_{n=1}^{\infty} |\langle f, e_n \rangle \cdot \overline{u_n}|
\cdot |r_{n}(t)|\qquad \forall\,\, t\in (a,b).
\end{equation}
From $(\ref{Vnt})$ and the Minkowski inequality, we see that $H$ belongs to the space $L^2 (a, b)$.
Clearly $|g_f(t)|\le H(t)$ for all $t\in (a,b)$, thus $g_f$ belongs to $L^2 (a, b)$ as well.

We also claim that $g_f$ converges in the $L^2 (a, b)$ norm. Indeed, we have
\[
\left|\left(g_f(t)-\sum_{n=1}^{M} V_n(t)\right)\right|^2=\left|\left(\sum_{n=M+1}^{\infty} V_n(t)\right)\right|^2 \le
\left(\sum_{n=M+1}^{\infty} |V_n(t)|\right)^2\le H^2(t).
\]
Moreover, $|g_f(t)-\sum_{n=1}^{M} V_n(t)|^2\to 0$ as $M\to\infty$
for all $t\in (a,b)$.
Then by the Lebesgue Convergence Theorem we have
\begin{equation}\label{tozeroU}
||g_f-\sum_{n=1}^{M}V_n||_{L^2 (a,b)}\to 0\qquad\text{as}\qquad M\to\infty.
\end{equation}

Finally, it is clear that each $V_n$ belongs to the space $\Clo (a,b)$.
Together with  $(\ref{tozeroU})$ yields that $g_f$  also belongs to the space $\Clo (a,b)$.\\

\smallskip

$(III)$ Next, let $h\in \Clo (a,b)$. It follows from $(\ref{rnkbound})$ and $(\ref{un})$ that
\[
|\overline{f(t)}\cdot\langle h, r_n\rangle\cdot u_n\cdot  e_n(t)|\le m_{\epsilon}e^{(-\delta+\epsilon)\lambda_n}
\cdot ||h||_{L^2 (a,b)}\cdot |f(t)|\qquad \text{for\,\, all}\quad t\in (a,b).
\]
This justifies below interchanging the integral sign with the summation sign
\begin{eqnarray}
\langle Th, f\rangle & = &
\int_{a}^{b} \overline{f(t)}\cdot \left(\sum_{n=1}^{\infty} \langle h, r_n\rangle\cdot u_n\cdot  e_n(t)\right)\, dt \nonumber\\
& = &
\sum_{n=1}^{\infty} \langle h, r_n\rangle\cdot u_n\cdot\int_{a}^{b} \overline{f(t)}\cdot e_n(t)\, dt \nonumber\\
& = & \sum_{n=1}^{\infty} \langle h, r_n\rangle\cdot u_n\cdot\langle e_n, f\rangle. \label{check}
\end{eqnarray}

Consider then the function $H$ as in $(\ref{H})$ with $H\in L^2 (a, b)$ as discussed before.
For any $M\in\mathbb{N}$, one has
\[
|\sum_{n=1}^{M} \langle f, e_n \rangle \cdot \overline{u_n}\cdot r_{n}(t)|\le |H(t)|\qquad \text{for\,\, all}\quad t\in (a,b).
\]
From this upper bound and the Lebesgue Convergence Theorem, we can interchange below the integral sign with the summation sign
to get
\begin{eqnarray}
\langle h, g_f\rangle & = & \int_{a}^{b} h(t)\cdot
\overline{\left(\sum_{n=1}^{\infty} \langle f, e_n \rangle \cdot \overline{u_n}\cdot r_{n}(t)\right)}\, dt\nonumber\\
& = &
\sum_{n=1}^{\infty} \langle e_n, f \rangle \cdot u_n\cdot\int_{a}^{b} h(t) \overline{r_{n}(t)}\, dt\nonumber\\
& = & \sum_{n=1}^{\infty} \langle e_n, f \rangle \cdot u_n\cdot\langle h, r_n\rangle. \label{check2}
\end{eqnarray}
But $(\ref{check2})$ is identical to $(\ref{check})$, thus we conclude that
$\langle Th, f\rangle = \langle h, g_f\rangle$. Clearly this holds for all $h\in \Clo (a,b)$ and therefore
\[
\langle h, T^*f\rangle = \langle h, g_f\rangle\qquad \text{for\,\, all}\quad h\in \Clo (a,b).
\]
By the uniqueness of the adjoint operator $T^*$, the formula $(\ref{adjoint})$ for $T^*$ is valid.
\end{proof}

\end{document}